\documentclass[11pt]{amsart}

\usepackage{enumerate, amsmath, amsthm, amsfonts, amssymb, 
mathrsfs, graphicx, paralist, ulem}
\usepackage[usenames, dvipsnames]{color}
\usepackage[margin=1in]{geometry} 
\usepackage{hyperref}
\usepackage{comment}

\usepackage{amsrefs}

\numberwithin{equation}{section}
\newtheorem{Theorem}[equation]{Theorem}
\newtheorem{Proposition}[equation]{Proposition}
\newtheorem{Lemma}[equation]{Lemma}
\newtheorem{Corollary}[equation]{Corollary}

\newtheorem{Question}[equation]{Question}

\theoremstyle{definition}
\newtheorem{Remark}[equation]{Remark}

\newtheorem{eg}[equation]{Example}

\newtheorem{Definition}[equation]{Definition}

\newcommand{\bA}{\mathbf{A}}

\newcommand{\bB}{\mathbf{B}}

\newcommand{\bG}{\mathbf{G}}

\newcommand{\cH}{\mathcal{H}}

\newcommand{\cO}{\mathcal{O}}

\newcommand{\cP}{\mathcal{P}}

\newcommand{\cT}{\mathcal{T}}

\newcommand{\bU}{\mathbf{U}}

\newcommand{\cW}{\mathcal{W}}

\newcommand{\bX}{\mathbf{X}}

\newcommand{\bY}{\mathbf{Y}}

\newcommand{\fg}{\mathfrak{g}}

\newcommand{\fu}{\mathfrak{u}}

\renewcommand{\phi}{\varphi}
\renewcommand{\emptyset}{\varnothing}

\renewcommand{\tilde}[1]{\widetilde{#1}}

\makeatletter
\def\Ddots{\mathinner{\mkern1mu\raise\p@
\vbox{\kern7\p@\hbox{.}}\mkern2mu
\raise4\p@\hbox{.}\mkern2mu\raise7\p@\hbox{.}\mkern1mu}}
\makeatother

\renewcommand{\hom}{\operatorname{Hom}}

\newcommand{\Waff}{W_{\text{aff}}}
\newcommand{\ZZ}{\mathbb{Z}}
\newcommand{\NN}{\mathbb{N}}
\newcommand{\CC}{\mathbb{C}}
\newcommand{\inv}[1]{\Delta(#1,-)}
\newcommand{\Inv}[1]{\Delta(#1,-)}
\newcommand{\GG}{\mathbb{G}}
\newcommand{\twomat}[4]{\begin{bmatrix}
        #1 & #2  \\[1em]
        #3 & #4 \\
     \end{bmatrix}}

\begin{document}

\title{On Iwahori-Hecke algebras for $p$-adic loop groups: double coset basis and Bruhat order }
\author{Dinakar Muthiah}
\maketitle

\begin{abstract}
We study the $p$-adic loop group Iwahori-Hecke algebra $\cH(G^+,I)$ constructed by Braverman, Kazhdan, and Patnaik in \cite{bkp} and give positive answers to two of their conjectures. First, we algebraically develop the ``double coset basis" of $\cH(G^+,I)$ given by indicator functions of double cosets. We prove a generalization of the Iwahori-Matsumoto formula, and as a consequence, we prove that the structure coefficients of the double coset basis are polynomials in the order of the residue field. The basis is naturally indexed by a semi-group $\cW_\cT$ on which Braverman, Kazhdan, and Patnaik define a preorder. Their preorder is a natural  generalization of the Bruhat order on affine Weyl groups, and they conjecture that the preorder is a partial order. We define another order on $\cW_\cT$ which is graded by a length function and is manifestly a partial order. We prove the two definitions coincide, which implies a positive answer to their conjecture. Interestingly, the length function seems to naturally take values in $\ZZ \oplus \ZZ \varepsilon$ where $\varepsilon$ is ``infinitesimally'' small.
\end{abstract}

\section{Introduction}

Let $\bG$ be a Kac-Moody group equipped with a choice of positive Borel subgroup $\bB$. Let $F$ be a local field, $\cO$ be its ring of integers, $\pi$ be a choice of uniformizer, and $k$ be the residue field of $\cO$. Let $G = \bG(F)$, let $K = G(\cO)$, and let the Iwahori subgroup $I$ be those elements of $K$ that lie in $\bB(k)$ modulo the uniformizer.

When $\bG$ is finite-dimensional, the Iwahori-Hecke algebra $\cH(G,I)$ is defined to be the set of complex valued functions on $G$ that are $I$-biinvariant and supported on finitely many $I$ double cosets. The multiplication in $\cH(G,I)$ is given by convolution. The $I$ double cosets of $G$ are indexed by the affine Weyl group. The ``double coset basis'' of $\cH(G,I)$ is given by indicator functions of $I$ double cosets, and the structure coefficients of this basis are given by the Iwahori-Matsumoto presentation of the algebra. Alternatively, Bernstein gave another presentation of $\cH(G,I)$ by making use of the principal series representation of $G$. In this presentation, $\cH(G,I)$ is generated by a finite Hecke algebra and the group algebra of the coweight lattice of $G$. 

In the case when $\bG$ is an untwisted affine Kac-Moody group, i.e. a ``loop group'', the definition of Iwahori-Hecke algebra due to Braverman, Kazhdan and Patnaik \cite{bkp} is more subtle. An initial issue is that the Cartan decomposition no longer holds. To handle this, let $G^+$ be the subset of $G$ where the Cartan decomposition {\it does} hold, and restrict attention to only those $I$-biinvariant functions whose support is contained in $G^+$. Then one can prove that $G^+$ is in fact a semi-group, and therefore the condition of having support in $G^+$ is preserved under convolution. Moreover, they prove that the convolution is well defined, i.e. the structure coefficients are finite. Also, the condition of being supported on finitely many cosets is preserved under convolution (i.e. one does not need to pass to a completion as one does in the spherical case \cite{bk,gr}). Let us write $\cH(G^+,I)$ for the \emph{$p$-adic loop group Iwahori-Hecke algebra} consisting of the set of complex-valued functions on $G^+$ supported on finitely many $I$ double cosets.  

Additionally, Braverman, Kazhdan, and Patnaik prove that $\cH(G^+,I)$ has a Bernstein-type presentation under which it is generated by an affine Hecke algebra and the semi-group algebra of the Tits cone of $\bG$. 
In this way, they show that $\cH(G^+,I)$ is a version of Cherednik's DAHA (see \cite{cher}). The only difference is that Cherednik's DAHA arises when one uses the coroot lattice of $G$ instead of the Tits cone in the Bernstein presentation (see Section \ref{sec:cherednik-daha}).

\subsection{The double coset basis}

There is another basis of  $\cH(G^+,I)$: the ``double coset basis'' given by indicator functions of $I$ double cosets. The $I$ double cosets contained in $G^+$ are naturally indexed by the semi-group $\cW_\cT$, which is the semi-direct product of the Weyl group of $G^+$ with the Tits cone. The stucture coefficients of this basis are given by the cardinalities of certain finite sets (see Theorem \ref{thm:loop-group-iwahori-hecke-algebra}) that arise from $p$-adic integration and are mysterious from an algebraic perspective. Braverman, Kazhdan, and Patnaik conjecture \cite[Section 1.2.4]{bkp} that there should be a combinatorial way to develop this basis; in particular, the structure constants of this basis should be polynomials in $q$, the order of the residue field $k$.

In this paper, we give a way to combinatorially develop the coset basis of $\cH(G^+,I)$. To do this, we prove a generalization of the Iwahori-Matsumoto relation (see Theorem \ref{thm:generalized-iwahori-matsumoto-formula} and its left-handed variation \ref{thm:generalized-iwahori-matsumoto-formula-left}) that holds in the case of loop groups.
Combining this with the algorithm developed in \cite[Section 6.2]{bkp} for writing the generators of the Bernstein presentation in terms of the coset basis we give a positive answer to the conjecture of Braverman, Kazhdan, and Patnaik:

\begin{Theorem}
The structure constants of the double coset basis of $\cH(G^+,I)$ are polynomials in $q$, the order of the residue field $k$.
\end{Theorem}

\subsection{The Bruhat order}

In the second part of this paper, we study candidates for the Bruhat order on $\cW_\cT$, which is the semi-group indexing the double coset basis of $\cH(G^+,I)$. One candidate is proposed in \cite[Section B.2]{bkp}. The authors define the notion of a \emph{double affine root}, and associated to such a root $\beta$, they define a reflection $s_\beta$.  If $w, ws_\beta \in \cW_\cT$, they declare that $w < ws_\beta$ if $w(\beta)$ is positive, and $w > ws_\beta$ otherwise. The Bruhat preorder is then defined to be the preorder generated by such inequalities. This definition generalizes a similar characterization of the Bruhat order for Weyl groups. However in the case of $\cW_\cT$, it is not at all clear that this preorder is in fact a partial order. Braverman, Kazhdan and Patnaik conjecture \cite[Section B.2]{bkp} that this preorder is a partial order.

We propose another candidate for the Bruhat order. The first ingredient is a new length function on $\cW_\cT$ whose definition (see Definition \ref{def:length-function}) is inspired by our generalized Iwahori-Matsumoto formula. We define an order generated by double affine reflections as above, but we say that $w < ws_\beta$ if the length of $ws_\beta$ is greater than the length of $w$. This order is manifestly a partial
order because it is graded by a length function. We then prove the following.

\begin{Theorem}
The two notions of Bruhat order agree. In particular, the Bruhat preorder considered by Braverman, Kazhdan, and Patnaik is in fact a partial order. This gives a positive answer to their conjecture.
\end{Theorem}

However, we must note that the length function used in the second partial ordering is not a naive generalization of the Coxeter length function which takes values in $\NN$. Instead, our length function take values in $\ZZ \oplus \ZZ \varepsilon$ ordered lexicographically (i.e. so $\varepsilon$ is ``infinitesimally small'' compared to an integer).

Because the ordinary length function for Weyl groups of Kac-Moody groups records the dimension of Schubert varieties, this seems to indicate that the dimensions of Schubert varieties in the ``double affine flag variety'' may naturally take values in $\ZZ \oplus \ZZ \varepsilon$. We do not currently have a good explanation for why this should be true geometrically, but it seems to indicate some very interesting phenomena.

\newcommand{\PP}{\mathbb{P}}
\subsection{Towards Kazhdan-Lusztig theory}
The longer term goal is to use the algebraic theory of $\cH(G^+,I)$ to understand the geometry of the yet-to-be-defined double affine flag variety. 

Thus to develop Kazhdan-Lusztig theory we need to accomplish the following tasks:
\begin{enumerate}
\item Explicitly understand the double coset basis. In particular, show that the structure constants depend polynomially on $q$.
\item Develop the strong Bruhat order.
\item Define and develop the Kazhdan-Lusztig involution.
\end{enumerate}

In this paper, we have made progress towards the first two tasks. What remains is to explicitly understand the Kazhdan-Lusztig involution, which we plan to address in a future paper. In finite type, this reduces to understanding $SL_2$, where the flag variety is $\PP^1$. However, in the double affine case we do not have such a simplification essentially because the double affine Weyl group is far from being a Coxeter group,. 

\subsection{The work of Bardy-Panse, Gaussent, Rousseau and generalizations}
We should mention that at the same time as this work, independent work by Bardy-Panse, Gaussent and Rousseau has appeared \cite{bgr}, which defines Iwahori-Hecke algebras in the general Kac-Moody case. Their main technical tool is the notion of a \emph{hovel}, a generalization of the notion of the affine building to Kac-Moody groups. We don't use hovels; instead, we make repeated use of the familiar Iwahori factorization to prove our computations. We both produce the same generalization of the Iwahori-Matsumoto formula \cite[Proposition 4.1]{bgr}. They do not however study the Bruhat order. 

Our results in Section \ref{sec:the-double-coset-basis} are currently stated only in the untwisted affine case, but the methods of proof are not specific to this case. We have restricted to this case because we use the results of Braverman, Kazhdan, and Patnaik for reasons related to the well-definedness of the algebra. If one knows well-definedness more generally, our proofs should work without modification.

\subsection{Acknowledgements}
I thank Alexander Braverman, Manish Patnaik, and Anna Pusk\'as for numerous fruitful conversations. 
\section{Preliminaries}

\subsection{Kac-Moody root data}
The following definitions are standard, and we refer the reader to \cite{kac,kumar,tits} for more details. Because we will mostly be working with coweights and coroots, we use the superscript $\vee$ to refer to weights and roots unlike the usual convention. 

Let $P$ be a finite-rank lattice, i.e. a finite-rank free abelian group, and let $P^\vee$ be its dual lattice. We will call $P$ the \emph{coweight lattice} and $P^\vee$ the \emph{weight lattice}. Let $I$ be a finite indexing set, and suppose we are given two embeddings
\begin{align}
  \label{eq:1}
  \alpha : I \hookrightarrow P \\
  \alpha^\vee : I \hookrightarrow P^\vee
\end{align}
For $i \in I$, we will follow the usual notation and write $\alpha_i$ (resp. $\alpha_i^\vee$) for $\alpha(i)$ (resp. $\alpha^\vee(i)$). Let $\Pi = \{ \alpha_i | i \in I \}$ and $\Pi^\vee = \{ \alpha^\vee_i | i \in I \}$. We call $\Pi$ (resp. $\Pi^\vee$) the set of \emph{simple coroots} (resp. \emph{simple roots}). A \emph{Kac-Moody root datum} $D$ is a tuple $(P,P^\vee,I,\alpha,\alpha^\vee)$ as above such that the matrix $A = ( \langle \alpha_i , \alpha^\vee_j \rangle ) _{i,j \in I}$ is a \emph{generalized Cartan matrix}. Let $Q$ be the sublattice of $P$ generated by $\Pi$, and let $Q^\vee$ be the sublattice of $P^\vee$ generated by $\Pi^\vee$. We call $Q$ (resp. $Q^\vee$) the \emph{coroot lattice} (resp. the \emph{root lattice}).

To each $i \in I$, we let $s_i$ be the linear automorphism of $P$ given by the following formula. 
\begin{align}
  s_i(\mu) = \mu - \langle \mu, \alpha_i^\vee \rangle \alpha_i
\end{align}
The \emph{Weyl group} $W$ of the root datum is defined to be the subgroup of $GL(P)$ generated by $\{s_i \mid i \in I\}$. It is known that $W$ is a Coxeter group.

There is an obvious notion of direct sum of root data, and we say that a root datum is \emph{irreducible} if it cannot be written as a direct sum of non-trivial root data.

\subsubsection{Fundamental (co)weights and $\rho^\vee$}
The \emph{dominant cone} $P^{++}\subset P$ is defined by the following.
\begin{align}
  \label{eq:4}
  P^{++} = \{ \lambda \in P | \langle \lambda, \alpha^\vee_i \rangle \geq 0 \text{ for all } i \in I \}
\end{align}
The \emph{Tits cone} $\cT \subset P$ is defined as $\cT = \bigcup_{w \in W} w(P^{++})$.

We say that a set $\{ \Lambda_i \mid i \in I \}$ of coweights indexed by $I$ is a set of \emph{fundamental coweights} if the following holds for all $i,j \in I$.
\begin{align}
  \label{eq:2}
 \langle \Lambda_i , \alpha^\vee_j \rangle = \delta_{i,j} 
\end{align}
Similarly we say that a set of weights $\{ \Lambda^\vee_i \mid i \in I \}$ is a set of \emph{fundamental weights} if the analagous statement holds with the positions of the superscript $\vee$ reversed.

Let us choose fundamental coweights and fundamental weights, and let us define 
\begin{align}
  \label{eq:3}
\rho^\vee = \sum \Lambda^\vee_i
\end{align}
Note that unlike in the case of a finite-dimensional Kac-Moody algebra (i.e. a semi-simple Lie Algebra), in general we must make a choice to define the fundamental coweights and weights because the simple coroots (resp. roots) do not form a basis of the coweight (resp. weight) space. 

\subsection{Kac-Moody groups}
To a Kac-Moody root datum $D$, Tits \cite{tits} associates a group functor $\bG_D$ on the category of commutative rings called the \emph{Kac-Moody group functor} associated to $D$. Unless the generalized Cartan matrix of $D$ is finite-type, this group functor is infinite-dimensional and will not be representable by a scheme. However, $\bG_D$ is representable by an \emph{affine group ind-scheme} of ind-finite type (see \cite{mathieu}). We will only refer to a single root datum at a time, so we will drop the subscript $D$. 

The group $\bG$ comes equipped with a pair of Borel subgroups $\bB^+$ and $\bB^-$. The subgroups $\bU^+$ and $\bU^-$ are their respective unipotent radicals, and the subgroup $\bA=\bB^+ \cap \bB^-$ is a finite-dimensional split torus. Note that we work with the {\it minimal Kac-Moody group}, so neither $\bB^+$ nor $\bB^-$ are completed.

We have natural identifications:
\begin{align} 
P = \hom(\GG_m, \bA)
\end{align}
and
\begin{align}
P^\vee = \hom(\bA,\GG_m)
\end{align}

\newcommand{\Lie}{\text{Lie}}
Moreover, we can identify $\bA= \bB^+/\bU^+ = \bB^-/\bU^-$.

\subsubsection{Roots and inversion sets.}
If we take points over $\CC$ (any characteristic zero field would do), we can look at $\bA(\CC)$ acting on $\fu^+ = \Lie(U^+)$ via the adjoint action. The set of \emph{positive roots} $\Delta_+$ is the set of weights for this action. Similarly, the \emph{negative roots} $\Delta_-$ are the weights for the action on $\fu^- = \Lie(U^-)$. By the construction of $\bG$, one sees that the simple roots are positive roots, and the set of \emph{real roots} are defined to be those roots obtained by translating simple roots by the Weyl group. 
\newcommand{\real}{\text{re}}
Let us write $\Delta_\real$ for the set of real roots, $\Delta_{+,\real}$ for the set of positive real roots, and $\Delta_{-,\real}$ for the set of negative real roots.

For each $w\in W$, define the \emph{inversion set} $\Inv{w}$ by
\begin{align}
  \label{eq:8}
  \Inv{w} = \{ \beta^\vee \in \Delta_{+} \mid w(\beta^\vee) \in \Delta_{-}\}
\end{align}

\subsubsection{Lifting Weyl group elements} 
For each $i \in I$, we have an \emph{$\mathbf{SL_2}$-root subgroup}
\begin{align}
  \label{eq:5}
  \phi_i : \mathbf{SL_2} \hookrightarrow \bG
\end{align}
Let us define 
\begin{align}
  \label{eq:6}
\overline{s_i} = \phi_i\left(\twomat{0}{-1}{1}{0}\right)
\end{align}
The map
\begin{align}
  \label{eq:7}
  s_i \mapsto \overline{s_i}
\end{align}
is a homomorphism from the braid group corresponding to $W$ to $\bG$. In particular, we can define $\overline{w} = \overline{s_{i_1}}\cdots \overline{s_{i_k}} \in \bG$ where $w = s_{i_1} \cdots s_{i_k}$ is a reduced decomposition. To declutter the notation we will omit the overline, and simply write $w \in \bG$ to denote $\overline{w}$.

\subsubsection{Steinberg relations}
To each real root $\beta$, there is an associated one-parameter subgroup
\begin{align}
  \label{eq:one-param-subgroups}\
  x_\beta : \GG_a \rightarrow \bG
\end{align}
If $\beta$ is positive, this morphism factors through $\bU^+$, and if $\beta$ is negative it factors through $\bU^-$.

Following \cite{tits} and \cite[Section 2.2.1]{bkp}, we say that a set $\Psi \subset \Delta_\real$ of real roots is \emph{pre-nilpotent} if there exist $w, w' \in W$ such that
\begin{align}
  w \Psi \subset \Delta_{+,\real} \\
  w' \Psi \subset \Delta_{-,\real} 
\end{align}
Given a pre-nilpotent pair $\{ \alpha, \beta \}$, set  $\theta(\alpha,\beta) = (\NN\alpha + \NN\beta) \cap \Delta_\real$. Then for any total order on $\theta(\alpha,\beta) - \{\alpha,\beta\}$, there exist a unique set of integers $k(\alpha,\beta ; \gamma)$ such that for any ring $S$ we have
\begin{align}
  \label{eq:steinberg-relations}
  x_\alpha(u)x_\beta(\tilde
u)x_\alpha(-u)x_\beta(-\tilde u) = \prod_{\gamma = m\alpha +n\beta \in \theta(\alpha,\beta) - \{\alpha,\beta\}} x_\gamma(k(\alpha,\beta ; \gamma) u^m {\tilde u}^n)
\end{align}
for all $u, \tilde{u} \in S$.

\subsubsection{Affine Kac-Moody group}
For our purposes, an \emph{untwisted affine Kac-Moody group} is a Kac-Moody group whose generalized Cartan matrix appears in the classification given in \cite[Chapter 4, Table Aff 1]{kac}. These groups are of central interest because they can be constructed from the loop groups of finite-type Kac-Moody groups. This relationship is well documented, so we refer the reader to \cite{kac,kumar,bkp} for details. There is a more general notion of affine Kac-Moody group that includes {\it twisted} loop groups. We do not address this case, so from now on we will simply write ``affine'' to mean ``untwisted affine''. Below we recall a few relevant facts about affine Kac-Moody root data.  

There is a canonical central cocharacter $\delta \in Q$, and a canonical imaginary root $\delta^\vee \subset Q^\vee$. We get a natural map:
\begin{align}
  \label{level}
  P \rightarrow \ZZ, \ \mu \mapsto \langle \mu, \delta^\vee \rangle
\end{align}
This is called the \emph{level} of the coweight. 

We write $P_k$ for the level-$k$ elements of $P$. In the affine case, we can describe the Tits cone explicitly
\begin{align}
  \label{eq:affine-Tits-cone}
  \cT = \cT_0 \oplus \bigoplus_{k > 0} P_k
\end{align}
where $\cT_0 = \{ r \delta \mid r \in \ZZ \}$.

\subsubsection{``Affine'' Weyl groups}

Because the Weyl group $W$ acts on the abelian group $P$ by automorphisms, we can form the semi-direct product 
\begin{align}
\cW_{P} = W \ltimes P
\end{align}
For $\mu \in P$, we denote by $\pi^\mu$ the corresponding element of $\cW_{P}$. The pair $(w,\mu) \in \cW_\cT$ will be written $w \pi^\mu$. 

One can easily verify that $Q \subset P$ and $\cT \subset P$ are each preserved under the Weyl group action. So we can also form:
\begin{align}
\cW_{Q} = W \ltimes Q
\end{align}
and
\begin{align}
\cW_\cT = W \ltimes \cT
\end{align}
Because $\cT$ is not closed under subtraction, $\cW_\cT$ is only a semi-group.

When $\bG$ is a simply-connected finite-type Kac-Moody group, we have 
\begin{align}
  \label{eq:affine-weyl-groups-in-finite-type}
  \cW_Q = \cW_P = \cW_\cT
\end{align}
but in general we have
\begin{align}
  \label{eq:affine-weyl-groups-in-general}
  \cW_Q \subset  \cW_P \supset \cW_\cT
\end{align}
In general, $\cW_Q$ and $\cW_\cT$ are not comparable.

When $\bG$ is affine type, $\cW_\cT$ has a natural ``level'' grading by non-negative integers where $\left(\cW_\cT\right)_n = \{ w \pi^\mu \in \cW_\cT \mid \text{level}(\mu) = n \}$. For each non-negative integer $n$, we say that $(\cW_\cT)_n$ is the set of elements in $\cW_\cT$ of \emph{level} $n$.

\subsection{Taking $p$-adic points}

\newcommand{\QQ}{\mathbb{Q}}
\subsubsection{Non-archimedean local fields}
Let $F$ be a non-archimedean local field. This means that $F$ is either isomorphic to the field of Laurent series over a finite field, or it is isomorphic to a finite extension of the field $\QQ_p$ of $p$-adic numbers.

Let $\cO$ be the ring of integers in $F$, let $\pi \in \cO$ be a uniformizing element, and let $k$ be the residue field of $\cO$. We let $q$ denote the cardinality of $k$. 

\subsubsection{Various subgroups of the $p$-adic group.}
We write $G = \bG(F)$. Abusing terminology, we call $G$ a \emph{$p$-adic group} even if $F$ has positive characteristic.
We write $K = \bG(\cO)$. We write $U^+_\cO = \bU^+(\cO)$, $U^-_\cO = \bU^-(\cO)$, $A_\cO =\bA(\cO)$, and $U^-_\pi = \{ u \in \bU^-(\cO) \mid u \equiv 1 \mod \pi \}$.

The Iwahori subgroup $I$ is defined as
\begin{align}
  \label{eq:iwahori-subgroup}
  I = \{ i \in K \mid i \in \bB^+(k) \mod \pi \}
\end{align}

We then have the following group decomposition known as the \emph{Iwahori factorization} (see \cite[Section 2]{im} and \cite[Section 3.1.2]{bkp}).
\begin{Proposition}
  \begin{align}
  I = U^+_\cO \cdot U^-_\pi \cdot A_\cO
  \end{align}
This also holds if we reorder the three factors in any way.
\end{Proposition}

We will also need the following lemma.

\begin{Lemma}
  \label{lem:UplusUminusK}
  \begin{align}
    \label{eq:15}
    \left(\bU^+(F) \bU^-(F)\right) \cap \bG(\cO) = \bU^+(\cO) \bU^-(\cO)
  \end{align}
\end{Lemma}

\begin{proof}
When $\bG$ is untwisted affine (which is the only case where we will actually need the lemma), this lemma is \cite[Appendix A.7]{bkp} (see also \cite[Lemma 3.3]{bgkp}). 

We give another argument that works in general. We claim the product $\bU^+ \bU^- \subset \bG$ is a closed sub-indscheme defined over $\ZZ$. For each antidominant weight $\lambda$, consider the integrable representation $L(\lambda)$ of lowest weight $\lambda$. It is known that this representation is defined over $\ZZ$. Let $v_\lambda$ be a lowest weight vector generating the lowest weight line, and let $v_\lambda^*$ be the covector in the dual representation of weight $-\lambda$ such that $\langle v_\lambda^*, v_\lambda \rangle = 1$. We consider the following function on $\bG$ 
\begin{align}
  \label{eq:16}
 \Delta_\lambda : g \mapsto \langle v_\lambda^*, g v_\lambda \rangle 
\end{align}
Using the Bruhat decomposition one can verify that, up to nilpotents, $\bU^+ \bU^-$ is cut out by the equations $\Delta_\lambda = 1$ as $\lambda$ varies over all anti-dominant weights.
In particular, regardless of nilpotents, we see that $\bU^+ \bU^-$ is a closed subscheme defined over $\ZZ$. 

We appeal to the following general fact: Let $B$ be a commutative ring, and let $A$ be a subring of $B$. Suppose $\bY$ is an affine scheme defined over $A$, and suppose $\bX \subset \bY$ is a closed subscheme defined over $A$. Then we have 
\begin{align}
  \label{eq:18}
 \bX(A) = \bX(B) \cap \bY(A) 
\end{align}

In particular, this also applies when $\bX$ and $\bY$ are ind-affine ind-schemes. As $\bG$ is an ind-affine ind-scheme, we apply this in the case of $\bU^+ \bU^- \subset \bG$.

\end{proof}

\subsubsection{Failure of the Cartan and Iwahori decomposition}

Recall that we identified $P$ with the cocharacter lattice of algebraic homomorphisms from $\GG_m$ to $\bA$. Taking $F$-points, for each $\mu \in P$, we obtain a group homomorphism 
\begin{align}
  \label{eq:9}
  F^* \rightarrow \bA(F)
\end{align}
We denote the image of $\pi$ under this map by $\pi^\mu$.

If $\bG$ is finite-type, i.e. it is a split semi-simple group, then we have the Cartan decomposition.
\begin{align}
G = \bigsqcup_{\lambda \in P^{++}} K \pi^\lambda K
\end{align}
However, Garland observed \cite{garland} that this is no longer true when $\bG$ is infinite-type. In this case, we define the following subset of $G$. 
\begin{Definition}
\begin{align}
G^+ = \bigsqcup_{\lambda \in P^{++}} K \pi^\lambda K
\end{align}
\end{Definition}
\begin{Theorem}{\cite{bk,garland},\cite[Appendix A]{bkp}}
 If $\bG$ is an untwisted affine Kac-Moody group, then $G^+$ is a sub-semi-group of $G$.
\end{Theorem}

If $\bG$ is finite-type, we also have the following Iwahori-decompostion
\begin{align}
G = \bigsqcup_{w \pi^\mu \in \cW_P} I w\pi^\mu I
\end{align}

Again this fails in infinite-type, but in affine type we have the following.
\begin{Proposition}{\cite[Proposition 3.4.2]{bkp}}
Suppose $\bG$ is untwisted affine type. Then we have  
  \begin{align}
  G^+ = \bigsqcup_{w \pi^\mu \in \cW_P} I w\pi^\mu I
  \end{align}
\end{Proposition}

\subsubsection{The $p$-adic loop group Iwahori-Hecke algebra.}

When $\bG$ is finite-type, the group $G$ acquires a natural topology under which it is locally compact. In particular, we can choose the Haar measure normalized so that $I$ has measure $1$. In this case, the Iwahori-Hecke algebra $\cH(G,I)$ is the space of compactly-supported complex-valued functions on $G$ that are biinvariant under $I$. The multiplication is convolution.

However, looking carefully at the definition, one can see that the existence of Haar measure are not necessary in order to define the convolution structure on $\cH(G,I)$. The compact-support condition is exactly the condition that a function be supported on finitely many $I$ double cosets, and the well-definedness of the multiplication corresponds exactly to the finiteness of of certain sets. The following is an easy exercise in $p$-adic integration (see, for example, \cite[Section 3.1]{im}). 

\begin{Proposition}
Let $\bG$ be a finite-type Kac-Moody group. For all $x \in \cW$, let $T_x$ be the indicator function of $I x I$ in $\cH(G,I)$ and write
\begin{align}
  \label{eq:iwahori-hecke-algebra-finite-type}
  T_x T_y = \sum_{z \in \cW} a^z_{x,y} T_z
\end{align}
then, 
\begin{align}
\label{eq:structure-coeffs}
a^z_{x,y} = | I \backslash \left(I  x^{-1} I z \cap I y I\right)|
\end{align}
In particular, the set of double cosets $I z I$ such that 
\begin{align}
I \backslash \left(I  x^{-1} I z \cap I y I\right) \neq \emptyset
\end{align}
is finite.
\end{Proposition}

When $\bG$ is of affine type, one uses \eqref{eq:structure-coeffs} as the definition of the convolution product. However, to obtain a well-defined multiplication, one needs to restrict to functions supported on $G^+$.
\begin{Definition}
Let $\bG$ be an untwisted affine Kac-Moody group, and let $G$ be the corresponding $p$-adic group. Then the \emph{Iwahori-Hecke algebra (for the  $p$-adic loop group $G$)} $\cH(G^+,I)$ is the vector space of complex-valued functions on $G^+$ that are supported on finitely-many double cosets.   
\end{Definition}

For all $x \in \cW_\cT$, let $T_x$ be the indicator function of $I x I$. Then it is clear that 
\begin{align}
  \label{eq:double-coset-basis}
  \{T_x \mid x \in \cW_\cT \}
\end{align}
is a basis for $\cH(G^+,I)$. We call this the \emph{double coset basis} of $\cH(G^+,I)$.

One of the main results of \cite{bkp} is the following theorem, which says that $\cH(G^+,I)$ has an algebra structure coming from convolution.
\begin{Theorem}{\cite[Theorem 5.2.1]{bkp}}
  \label{thm:loop-group-iwahori-hecke-algebra} Let $\bG$ be an untwisted affine Kac-Moody group, and let $x, y \in \cW_\cT$. Then for all $z \in \cW_\cT$, the set
\begin{align}
I \backslash \left(I  x^{-1} I z \cap I y I\right)
\end{align}
is finite. Let $a^z_{x,y}$ be the cardinality of this set. For all but finitely many $z \in \cW_\cT$, we have $a^z_{x,y}=0$, and the formula
\begin{align}
  \label{eq:iwahori-hecke-algebra-affine-type}
  T_x T_y = \sum_{z \in \cW} a^z_{x,y} T_z
\end{align}
defines an associative algebra structure on $\cH(G^+,I)$.
\end{Theorem}

\subsection{Various versions of the Double Affine Hecke Algebra.}

\subsubsection{Coxeter-Hecke Algebras}

Let $W$ be a Coxeter group with simple reflections $\{ s_i \mid i \in I \}$ where $I$ is some indexing set. To $W$ we can associate a corresponding Hecke algebra $\cH_W$, which is the algebra over $R = \CC[v,v^{-1}]$ generated by symbols $T_w$ for $w \in W$ subject to the following relations.

\begin{itemize}
\item  $T_{w_1}T_{w_2} = T_{w_1 w_2}$ if $\ell(w_1 w_2) = \ell(w_1) + \ell(w_2)$ where $\ell$ is the usual length function on a Coxeter group.
  \item $(T_{s_i} +1) (T_{s_i} - v^2) = 1$ for all simple reflections $i \in I$.
\end{itemize}
We will follow the usual convention and write $T_i$ for $T_{s_i}$ when $i \in I$.

\newcommand{\HH}{\mathbb{H}}

\subsubsection{The Garland-Gronowski DAHA}
Let $R = \CC[v,v^{-1}]$. Consider the following $R$-module.
\begin{align}
\HH = \cH_W \otimes_R R[P]
\end{align}
For $\mu \in P$, let us write $\Theta_\mu$ for the element $1 \otimes \mu \in \HH$.

Then following Garland-Gronowski \cite{gg} and \cite[Section 5.1]{bkp} we give an algebra structure to $\HH$ by requiring that
\begin{itemize}
\item $\cH_W \otimes 1$ be a copy of the Coxeter-Hecke algebra,
\item $1 \otimes R[P]$ be a copy of the group algebra $R[P]$,
\item the Bernstein relation:
  \begin{align}
    \label{eqn:bernstein-relation}
    T_i \Theta_\mu - \Theta_{s_i(\mu)}T_i = (v^{-2} - 1) \frac{\Theta_{\mu} -\Theta_{s_i(\mu)}}{1 - \Theta_{-\alpha_i}}
  \end{align}
\end{itemize}

When $\bG$ is affine, $\HH$ carries a natural $\ZZ$ grading where $\cH_W$ has degree $0$, and $\deg \Theta_\mu = \text{level}(\mu)$.

\subsubsection{Cherednik's DAHA and Tits DAHA}{\label{sec:cherednik-daha}}

The subspace $\HH_Q = \cH_W \otimes_R R[Q] \subset \HH$ is a subalgebra. 

When $\bG$ is untwisted affine, then $\HH_Q=\HH_0$ (the degree-$0$ part of $\HH$ under the level grading) is naturally isomorphic to Cherednik's double affine Hecke algerbra \cite{cher} (Cherednik's parameter $t$ corresponds to $v^{-2}$, and the parameter $q$ corresponds to the central element $\Theta_\delta$).

We can also form the subalgebra
\begin{align}
  \label{eq:tits-daha}
  \HH_\cT = \cH_W \otimes_R R[\cT]
\end{align}
We propose that when $\bG$ is affine that this algebra be called the \emph{Tits DAHA}.

\subsubsection{The relationship with $\cH(G^+,I)$.}

The following result is due to Braverman, Kazhdan, and Patnaik.
\begin{Theorem}{\cite[Theorem 5.34]{bkp}}
 When $\bG$ is affine, there is an algebra isomorphism 
between the Tits DAHA specialized at $v = q^{-1/2}$ and the $p$-adic loop group Iwahori-Hecke algebra.
\begin{align}
\phi : {\HH_\cT}|_{v = q^{-1/2}} \rightarrow \cH(G^+,I)
\end{align}
\end{Theorem}

We recall the main properties of this isomorphism. First, for $w \in W$, we have $\phi(T_w) = T_w$. Second, for $\lambda$ dominant $\phi(\Theta_\lambda) = q^{\langle \rho^\vee, \lambda \rangle} T_{\pi^\lambda}$.
Finally, for general $\mu \in \cT$ the authors provide an explicit algorithm for writing $\phi(\Theta_\mu)$ in terms of the double coset basis with coefficients that are Laurent polynomials in $q$ with integer coefficients (see \cite[Section 6.2]{bkp}).

\subsection{Preorders and partial order.}

Recall that a \emph{preorder} on a set $X$ is a binary relation $\leq$ on $X$ satisfying the following properties.
\begin{itemize}
\item For all $x \in X$, $x \leq x$.
\item If $x \leq y$ and $y \leq z$, then $x \leq z$.
\end{itemize}
We write $x < y$ to mean $x \leq y$ and $x \neq y$.
We furthermore say that $\leq$ is a \emph{partial order} if the following property holds: suppose $x,y\in X$ are such that $x\leq y$ and $y \leq x$, then $x = y$.

Suppose $X$ and $Y$ are both preordered sets. Then we say that a map 
\begin{align}
 \ell : X \rightarrow Y
\end{align}
is a \emph{grading} if 
\begin{align}
  \ell(x_1) < \ell(x_2) \text{ whenever } x_1 < x_2
\end{align} 
We then have the following lemma.
\begin{Lemma}
 Suppose that $X$ is a preordered set, that $Y$ is a partially ordered set, and that $ \ell : X \rightarrow Y$ is a grading. Then the preorder on $X$ is a partial order.
\end{Lemma}

\section{The double coset basis}{\label{sec:the-double-coset-basis}}

The algebra $\cH(G^+,I)$ has two natural bases; there is the ``Bernstein basis'' $\{\Theta_\mu T_w \mid \pi^\mu w \in W_\cT \}$ and the double coset basis $\{T_{\pi^\mu w} \mid \pi^\mu w \in W_\cT \}$. 
Using the Bernstein relation \eqref{eqn:bernstein-relation}, it is easy to see that the structure coefficients of the Bernstein basis are Laurent polynomials in $q$. Furthermore, Braverman, Kazhdan and Patnaik \cite[Section 6.2]{bkp} provides an algorithm to write the Bernstein basis in terms of the double coset basis. 
From this algorithm, one can see that the coefficients of the Bernstein basis vectors, when written in the double coset basis, are Laurent polynomials in $q$. 
One of the results of this section is an inverse algorithm. We will develop the double coset basis combinatorially, and as a consequence we will see that the coefficients of the double coset basis when written in the Bernstein basis are Laurent polynomials in $q$. As a corollary, we see that the structure coefficients of the double coset basis are Laurent polynomials in $q$. Because these structure coefficients are known to be integers for all $q$ that are prime powers, we can conclude that the structure coefficients are in fact ordinary polynomials in $q$.

\subsection{The Iwahori-Matsumoto relation}

In $\cH(G^+,I)$, we have the following relations.

\begin{Theorem}\label{thm:generalized-iwahori-matsumoto-formula}
Let $\mu \in \cT$ be a Tits coweight, $w \in W$ be an element of the single affine Weyl group, and let $i \in I$ be a node of the single affine Dynkin diagram. Then:

\begin{align}{\label{eqn:generalized-iwahori-matsumoto-formula}}
T_{\pi^{\mu} w s_i} = \left\{
\begin{array}{c l}      
   T_{\pi^\mu w} T_i \text{ if } \langle \mu, w(\alpha_i) \rangle > 0 \text{ or if }   \langle \mu, w(\alpha_i) \rangle = 0 \text{ and } w(\alpha_i) > 0\\
    T_{\pi^\mu w} T_i^{-1}  \text{ if } \langle \mu, w(\alpha_i) \rangle < 0 \text{ or if }   \langle \mu, w(\alpha_i) \rangle = 0 \text{ and } w(\alpha_i) < 0
\end{array}\right.
\end{align}

\end{Theorem}

\begin{Proposition}
Let us suppose the setup of the above theorem.  

If  $\langle \mu, w(\alpha_i) \rangle > 0$ or if $\langle \mu, w(\alpha_i) \rangle = 0 \text{ and } w(\alpha_i) > 0$, then 
\begin{align}
I \pi^{\mu} w I s_i I = I \pi^{\mu} w s_i I
\end{align}
If $\langle \mu, w(\alpha_i) \rangle < 0$ or if  $\langle \mu, w(\alpha_i) \rangle = 0 \text{ and } w(\alpha_i) < 0$
\begin{align}
I \pi^{\mu} w s_i I s_i I = I \pi^{\mu} w I
\end{align}
\end{Proposition}

\begin{proof}
For the first equation, we calculate:

\begin{align}
I \pi^\mu w I s_i I = I \pi^\mu w \cdot x_{\alpha_i}(\cO) \cdot s_i I = I \pi^\mu x_{ w(\alpha_i)}(\cO) w s_i I = \\ I  x_{ w(\alpha_i)}( \pi^{\langle \mu, w \alpha_i \rangle} \cO) \pi^\mu w s_i I = I \pi^\mu w s_i I
\end{align}

The first equality comes from the Iwahori factorization and Bruhat decompositions. The last equality follows because of the assumption that $\langle \mu, w(\alpha_i) \rangle > 0$ or $\langle \mu, w(\alpha_i) \rangle = 0 \text{ and } w(\alpha_i) > 0$ .

For the second equation, we calculate:

\begin{align}
I \pi^{\mu} w s_i I s_i I =I \pi^{\mu} w s_i x_{\alpha_i}(\cO) s_i I = I \pi^{\mu} x_{- w \alpha_i} (\cO) w I = \\ I x_{- w(\alpha_i)}( \pi^{\langle \mu,  -w \alpha_i \rangle} \cO) \pi^\mu w I = I \pi^{\mu} w I
\end{align}

Where the last equality follows because of our assumptions on $\mu, w, \alpha_i$.

\end{proof}

This proves \eqref{eqn:generalized-iwahori-matsumoto-formula} is true up to a constant. So all that remains is showing that the constant is $1$.

\begin{proof}[Proof of Theorem \ref{thm:generalized-iwahori-matsumoto-formula} ] 
 Let $\mu \in \cT$, $w \in W$, $i \in I$. Let us consider the case when $w(\alpha_i)$ is positive and $\langle w^{-1}(\mu), \alpha_i \rangle \geq 0$. The other cases are similar.

It suffices to show that 
  \begin{align}
    I \backslash (I w^{-1}\pi^{-\mu}I \pi^\mu w s_i \cap Is_iI )
  \end{align}
is a point.
On the one hand, we have:  $Is_iI=Is_ix_{\alpha_i}(\cO)$. By the Iwahori factorization, we have 
\begin{align}
 I w^{-1}\pi^{-\mu}I \pi^\mu w s_i=I w^{-1}\pi^{-\mu}U_\cO U^-_\pi \pi^\mu w s_i   
\end{align}
We need to consider all $i\in I$, $u_+\in U_\cO$, $u_- \in U^-_\pi$, and $f \in \cO$ such that
\begin{align}
 iw^{-1}\pi^{-\mu}u_+u_-\pi^\mu w s_i = s_i x_{\alpha_i}(f)   
\end{align}
Because $I \backslash I s_i x_{\alpha^\vee_i}(\pi \cO)$ is a point, it will suffice to show that $f \in \pi \cO$. Also note that $\pi^{-\mu}u_+\pi^\mu \in U_\cO$ and $\pi^{-\mu}u_-\pi^\mu\in U^-_\cO$ by Lemma~\ref{lem:UplusUminusK}.

Moreover, we can factorize $u_+ = u_1 u_2$ where $w^{-1} u_1 w \in U$ and $w^{-1} u_2 w \in U^-$. In particular, $\pi^{-w^{-1}(\mu)}w^{-1} u_1 w\pi^{w^{-1}(\mu)} \in I$. We can also factorize $u_-=u_3 u_4$ where $s_iw^{-1}u_3ws_i \in U^-$ and $s_iw^{-1}u_4ws_i \in U^+$. We can further factorize $u_4 = u_5 x_{-w(\alpha_i)}(g)$ where $w^{-1}u_5w \in U^+$ and $g \in \pi\cO$ (it is here that we use the assumption that $w(\alpha_i)$ is positive). 
So we then have the following. 
\begin{align}
 \left(\pi^{-w^{-1}(\mu)} w^{-1} u_2 w \pi^{w^{-1}(\mu)}\right) 
\left(\pi^{-w^{-1}(\mu)}w^{-1}u_3w  \pi^{-w^{-1}(\mu)}\right) 
\left(\pi^{-w^{-1}(\mu)}w^{-1}u_5w \pi^{w^{-1}(\mu)}\right) 
x_{-\alpha_i}(\pi^{\langle w^{-1}(\mu),\alpha_i\rangle}g-f)  \in I 
\end{align}
By the Steinberg relations ~\eqref{eq:steinberg-relations}, when we commute 
$\left(\pi^{-w^{-1}(\mu)}w^{-1}u_5w \pi^{w^{-1}(\mu)}\right)$ past $\left(x_{-\alpha_i}(\pi^{\langle w^{-1}(\mu),\alpha_i\rangle}g-f)\right)$, we only get terms in $U^+_\cO$. In particular, they lie in $I$. So we see the following.
\begin{align}
   \left(\pi^{-w^{-1}(\mu)} w^{-1} u_2 w \pi^{w^{-1}(\mu)}\right) 
\left(\pi^{-w^{-1}(\mu)}w^{-1}u_3w  \pi^{-w^{-1}(\mu)}\right) 
x_{-\alpha_i}(\pi^{\langle w^{-1}(\mu),\alpha_i\rangle}g-f)
  \in I
\end{align}
Because the first two terms lie in $\{ u \in U^- \mid s_i u s_i \in U^- \}$, we must have $
\pi^{\langle w^{-1}(\mu),\alpha_i\rangle}g-f \in \pi\cO$. As we have assumed ${\langle w^{-1}(\mu),\alpha_i\rangle} \geq 0$, we must have $f \in \pi\cO$. 
\end{proof}

We also have the following left-hand version of the Iwahori-Matsumoto formula, whose proof is analogous to the right-hand version.
\begin{Theorem}{(Left-handed version of Theorem \ref{thm:generalized-iwahori-matsumoto-formula})} \label{thm:generalized-iwahori-matsumoto-formula-left}
Let $\mu \in \cT$ be a Tits coweight, $w \in W$ be an element of the single affine Weyl group, and let $i \in I$ be a node of the single affine Dynkin diagram. Then:
\begin{align}{\label{eqn:generalized-iwahori-matsumoto-formula-left}}
T_{s_i \pi^{\mu} w } = \left\{
\begin{array}{c l}      
  T_i T_{\pi^\mu w}  \text{ if } \langle \mu,\alpha_i \rangle > 0 \text{ or if }   \langle \mu,\alpha_i \rangle = 0 \text{ and } w^{-1}(\alpha_i) > 0\\
   T_i^{-1} T_{\pi^\mu w}   \text{ if } \langle \mu, \alpha_i \rangle < 0 \text{ or if }   \langle \mu, \alpha_i \rangle = 0 \text{ and } w^{-1}(\alpha_i) < 0
\end{array}\right.
\end{align}
\end{Theorem}

With these formulas, we deduce the folowing formula for those double coset basis elements corresponding to arbitrary coweights in the Tits cone.
\begin{Corollary}
Let $w \in W$, and let $\lambda$ be a dominant coweight, then we have
\begin{align}
  \label{eq:non-dominant-coset-basis}
T_{\pi^\lambda} T_{w^{-1}} = T_{\pi^\lambda w^{-1}} =T_{w^{-1}} T_{\pi^{w\left(\lambda\right)}} 
\end{align}
In particular, this implies
\begin{align}
T_{\pi^{w\left(\lambda\right)}}  = T_{w^{-1}}^{-1}T_{\pi^\lambda} T_{w^{-1}}
\end{align}
\end{Corollary}

Recalling that $\Theta_\lambda = q^{\langle \rho^\vee, \lambda \rangle} T_{\pi^\lambda}$ for dominant coweights $\lambda$, we see that when one writes double coset basis elements in terms of the Bernstein basis, the coefficients are Laurent polynomials in $q$. Therefore, as discussed at the beginning of this section, we can conclude that the structure coefficients for the double coset basis are Laurent polynomials in $q$. Because we know that the these structure coefficients always specialize to non-negative integers when $q$ is a prime power, we can in fact conclude the following.
\begin{Theorem}
The structure coefficients of the double coset basis are polynomials in $q$.
\end{Theorem}

\section{Bruhat orders and the enhanced length function}{\label{sec:bruhat-order-section}}
The results of this section hold for any Kac-Moody group $\bG$, but we will be most interested in the case when $\bG$ is affine type. We use the adjective ``double-affine'' to refer to many of the concepts considered in this section, but we caution that this terminology is only really appropriate when $\bG$ is affine-type.

\subsection{Double affine roots and reflections}

Let us consider the space $Q^\vee \oplus \ZZ \pi$, which we can think of as the ``double affine root lattice". We say an element $\beta^\vee + n \pi$ is a \emph{(real) double affine root} if $\beta^\vee$ is a real affine root for $\fg$. We say that $\beta^\vee + n \pi$ is a positive double affine real root if $\beta^\vee > 0$ and $n \geq 0$ or $\beta^\vee < 0$ and $n > 0$. 

\begin{Definition}
Let $\beta^\vee + n \pi$ be a positive double affine root. We define the associated \emph{reflection} as follows
\begin{align}
s_{\beta^\vee + n \pi} =   \left\{
\begin{array}{c l}      \pi^{n\beta} s_{\beta}  \text{ if } \beta^\vee > 0 \\
\pi^{-n\beta} s_{\beta} \text{ if } \beta^\vee < 0
\end{array}\right.
\end{align}
\end{Definition}

Note that this element lies in the double affine Weyl group $W_Q$, but not in the Tits double affine Weyl group. 

We define an action of $\cW_\cP$ on $Q^\vee \oplus \ZZ \pi$ as follows:
\begin{align}
\pi^\mu w ( \gamma + n \pi) = \pi^\mu( w(\gamma) + n \pi)  = w(\gamma) + (n + \langle \mu, \gamma \rangle) \pi
\end{align}
\begin{Remark}
  This definition is a verbatim generalization of the notion of affine real root and affine reflections when $G$ is a finite-type Kac-Moody group.
\end{Remark}

\subsection{The Bruhat preorder defined by Braverman, Kazhdan, and Patnaik}
In \cite[Section B.2]{bkp}, the authors define a preorder on $\cW_\cT$ as follows. Let $x,y \in \cW_\cT$, and suppose that there is a positive double affine root $\beta^\vee +n \pi$ such that 
\begin{align}
  x = y s_{\beta^\vee +n \pi} 
\end{align}
and
\begin{align}
  y(\beta^\vee +n \pi) \text{ is positive}
\end{align}
Then we say that $y \leq x$, and we say the \emph{(first) Bruhat preorder} $<$ on $W_\cT$ is the preorder generated by all such inequalities. It isn't clear from the definition that this preorder is in fact an order, but the authors of \cite{bkp} conjecture it to be so.

\begin{Remark}
The definition above is slightly different than that given by Braverman, Kazhdan, and Patnaik. They define a preorder on all of $\cW_P$ using the above formulas, and restrict this order to $\cW_\cT$. The most interesting situation is for elements of strictly positive level; here the orders coincide because the positive level elements of $\cW_P$ and $\cW_\cT$ coincide. For elements of level zero, however, it is not clear whether the two orders coincide. 

But we believe that the definition given by working in $\cW_P$ and then restricting to $\cW_\cT$ is unnatural. The level-zero elements of $\cW_\cT$ are isomorphic to the product of $W$ and a copy of $\ZZ$ corresponding to the central cocharacter. In this case, we would expect the Bruhat order on each subset $W \times \{n\}$ to be isomorphic to the Bruhat order on $W$. The definition above gives exactly this order for level-zero elements.  
\end{Remark}

\begin{Remark}
 Also, the definition considered in \cite{bkp} involves a right action of $W_\cT$ on double affine roots, but it is easy to check that it is equivalent to the one we consider.
\end{Remark}

\subsection{Length function and another Bruhat order}
Let us define the \emph{length function} $\ell$ as follows. Lengths take values in $\ZZ \oplus \ZZ \varepsilon$, which we order lexicographically. Here $\varepsilon$ is a formal symbol which we can think of as being infinitesimally smaller than one, i.e. we have $n \varepsilon < 1$ for any integer $n$.

When $\lambda$ dominant, we define :
\begin{align}
\ell(\pi^\lambda) = 2 \langle \lambda, \rho^\vee \rangle 
\end{align}
For general $\mu \in \cT$, pick $w \in W$ so that $w(\mu)$ is dominant. Then we make the following definition.
\begin{align}
\ell(\pi^\mu) = 2 \langle w(\mu), \rho^\vee \rangle 
\end{align}

The next proposition follows immediately from the definition of $\ell$ and the property that for all $w \in W$ we have:
\begin{align}
w(\rho^\vee)= \rho^\vee - \sum_{\beta^\vee \in \Inv{w}} \beta^\vee 
\end{align}
\begin{Proposition}
  \label{prop:maximum-property-of-big-lengths}
For any Tits coweight $\mu$, we have:
\begin{align}
\ell(\pi^\mu) = \max_{w \in W} 2 \langle w (\mu), \rho^\vee \rangle
\end{align}
\end{Proposition}

\begin{Definition}
  \label{def:length-function}
We define the length function $\ell : \cW_\cT \rightarrow \ZZ \oplus \ZZ \varepsilon$ as follows. For $\mu \in \cT$, we define $\ell(\pi^\mu)$ using the above formulas, and for general elements $\pi^\mu w \in \cW_\cT$ we define:
\begin{align}
\ell(\pi^\mu w) = \ell(\pi^\mu) + \varepsilon \cdot \left(  | \{ \beta^\vee \in \inv{w^{-1}} : \langle \mu , \beta^\vee \rangle \geq 0 \} | - | \{ \beta^\vee \in \inv{w^{-1}} : \langle \mu , \beta^\vee \rangle < 0 \} | \right)
\end{align}
\end{Definition}
\newcommand{\ellbig}{\ell_{\text{big}}}
\newcommand{\ellsmall}{\ell_{\text{small}}}
\begin{Definition}
Suppose $\pi^\mu w \in W_\cT$. Then we can write $\ell(\pi^\mu w) = \ellbig(\pi^\mu w) + \ellsmall(\pi^\mu w)\varepsilon$. We call   $\ellbig$ the \emph{big length}  and $\ellsmall$ the \emph{small length}.
\end{Definition}

\begin{Lemma}
  \label{lem:length-recursion-right}
The length function satisfies the following recursive relation.
\begin{align}
\ell(\pi^\mu w s_i) =  \left\{
\begin{array}{c l}      
   \ell(\pi^\mu w) + \varepsilon \text{ if } \langle \mu, w(\alpha_i) \rangle > 0 \text{ or if }   \langle \mu, w(\alpha_i) \rangle = 0 \text{ and } w(\alpha_i) > 0\\
     \ell(\pi^\mu w) - \varepsilon  \text{ if } \langle \mu, w(\alpha_i) \rangle < 0 \text{ or if }   \langle \mu, w(\alpha_i) \rangle = 0 \text{ and } w(\alpha_i) < 0	
\end{array}\right.
\end{align}
\end{Lemma}

Note that the dichotomy of this recurrence is precisely the dichotomy of the generalized Iwahori-Matsumoto relations for the Tits DAHA that we produced in the previous section.


We also have the following left-hand version of the above recursion relation.
\begin{Lemma}
  \label{lem-length-recursion-left}
The length function satisfies the following recursive relation.
\begin{align}
\ell(s_i \pi^\mu w ) =  \left\{
\begin{array}{c l}      
   \ell(\pi^\mu w) + \varepsilon \text{ if } \langle \mu, \alpha_i \rangle > 0 \text{ or if }   \langle \mu, \alpha_i \rangle = 0 \text{ and } w^{-1}(\alpha_i) > 0\\
     \ell(\pi^\mu w) - \varepsilon  \text{ if } \langle \mu, \alpha_i \rangle < 0 \text{ or if }   \langle \mu, \alpha_i \rangle = 0 \text{ and } w^{-1}(\alpha_i) < 0	
\end{array}\right.
\end{align}
\end{Lemma}

\begin{Definition}
Let $x,y \in W_\cT$, and suppose that there is a positive double affine root $\beta^\vee +n \pi$ such that 
\begin{align}
  x = y s_{\beta^\vee +n \pi} 
\end{align}
and
\begin{align}
  \ell(x) > \ell(y)
\end{align}
Then we write $y \preceq x$, and we say the \emph{(second) Bruhat order} $\prec$ on $W_\cT$ is the preorder generated by such inequalities. Unlike in the case of the first Bruhat order, it is manifestly clear that $\prec$ is a partial order because it is graded by the length function. 
\end{Definition}

\subsection{Proving that the two Bruhat orders coincide}
The rest of this section is devoted to proving the following theorem.
\begin{Theorem}
  \label{thm:the-two-orders-coincide}
  The two orders $<$ and $\prec$ coincide. 
\end{Theorem}
\begin{Remark}
  In particular,  we see that the preorder $<$ is a partial order, which gives a positive answer to a conjecture of Braverman, Kazhdan, and Patnaik \cite[Section B.2]{bkp}.
\end{Remark}
\begin{Lemma}
  \label{lem:big-length-lemma}
Let $\nu \in \cT$, let $\beta^\vee$ be a positive real root, and let $\beta$ be the corresponding coroot. Suppose $\langle \nu, \beta^\vee \rangle > 0 $. Then for integers $m$ such that $0 < m < \langle \nu, \beta^\vee \rangle$ and $\nu - m\beta \in \cT$, we have:
\begin{align}
\ell(\pi^{\nu - m \beta}) < \ell(\pi^\nu)
\end{align}
\end{Lemma}

\begin{proof}
Because the length function is invariant for lattice elements under conjugation by $W$, we can assume $\beta$ is a simple coroot $\alpha_i$ (choose $w$ that sends $\beta$ to $\alpha_i$, and replace $\nu$ by $w(\nu)$).

Then we claim 
\begin{align}
\langle \nu - m \alpha_i, v(\rho^\vee) \rangle < \langle \nu, v(\rho^\vee) \rangle
\end{align}
for all $v \in W$.

There are two cases, depending on whether $v^{-1}( \alpha_i)$ is positive or negative.

If $v^{-1}( \alpha_i)$ is positive, then applying $v^{-1}$ and using the fact that $m > 0$, we have the inequality.

If $v^{-1}( \alpha_i)$ is negative, then we can write $v = s_i u$, where $u^{-1}(\alpha_i)$ is positive. In this case:
\begin{align}
\langle \nu - m \alpha_i, v(\rho^\vee) \rangle  = \langle \nu - m \alpha_i, s_i u(\rho^\vee) \rangle = \langle \nu + (m-\langle \nu, \alpha_i \rangle) \alpha_i,  u(\rho^\vee) \rangle
\end{align}
Because $(m-\langle \nu, \alpha_i \rangle) < 0$, we can argue as we did in the first case.

\end{proof}


\begin{Lemma}{\label{lem:big-length-lemma-2}} Let $\mu \in \cT$, and let $\beta$ be a positive affine coroot such that $\langle \mu, \beta^\vee \rangle \neq 0$. Suppose $k \in \ZZ$ is such that $\mu - k \beta \in \cT$.
Let $t = \frac{k}{\langle k, \beta^\vee \rangle}$, which is the unique real number satisfying 
\begin{align}
  \label{eq:10}
  \mu - k \beta = (1-t) \cdot \mu + t \cdot s_\beta(\mu)
\end{align}
If $0 < t < 1$, then we have:
\begin{align}
\ell(\pi^{\mu - k \beta}) < \ell(\pi^\mu)
\end{align}
If $t<0$ or $t>1$, then we have:
\begin{align}
\ell(\pi^{\mu - k \beta}) > \ell(\pi^\mu)
\end{align}
Of course, if $t=0$ or $t=1$, we have:
\begin{align}
\ell(\pi^{\mu - k \beta}) = \ell(\pi^\mu)
\end{align}
\end{Lemma}
\begin{proof}
The various cases can be handled by applying Lemma~\ref{lem:big-length-lemma} using the following particular choices of $\nu$ and $m$.
\begin{itemize}
\item If $\langle \mu, \beta^\vee \rangle > 0$ and $0 < k < \langle \mu, \beta^\vee \rangle$, use $\nu = \mu$ and $m = k$.
\item If $\langle \mu, \beta^\vee \rangle > 0$ and $k > \langle \mu, \beta^\vee \rangle$, use $\nu = s_\beta(\mu) + k \beta$ and $m = k-\langle \mu, \beta \rangle$. 
\item If $\langle \mu, \beta^\vee \rangle < 0$ and $0 > k > \langle \mu, \beta^\vee \rangle$, use $\nu = s_\beta(\mu)$ and $m =k -\langle \mu, \beta^\vee \rangle -k$.
\item If $\langle \mu, \beta^\vee \rangle < 0$ and $k < \langle \mu, \beta^\vee \rangle$, use $\nu = \mu-k\beta$ and $m = -k$.
\end{itemize}
\end{proof}

\begin{Lemma}
  \label{lem:inversion-set-lemma}
 Let $\beta^\vee$ be a positive (single affine) real root, and let $\mu$ be a coweight. Then
\begin{align}
 \left|\{ \gamma^\vee \in \Inv{s_{\beta}} : \langle \mu, \gamma^\vee \rangle \geq 0 \}\right| - \left|\{ \gamma^\vee \in \Inv{s_{\beta}} : \langle \mu, \gamma^\vee \rangle < 0 \}\right| 
\end{align}
is strictly positive if and only if 
\begin{align}
  \langle \mu, \beta^\vee \rangle \geq 0
\end{align}
\end{Lemma}
\begin{proof}
Consider the involution $\iota$ of $\Inv{s_\beta}$ defined by the following formula.
  \begin{align}
    \iota(\gamma^\vee) = -s_{\beta}(\gamma^\vee)
  \end{align}
It is easy to see that the only fixed point of $\iota$ is $\beta^\vee$.
In particular, we see that $\Inv{s_\beta}$ has odd order.
Suppose $ \langle \mu, \beta^\vee \rangle \geq 0$.
Let $\gamma^\vee \in \Inv{s_\beta}$. 
Then we must have $\langle \gamma^\vee, \beta \rangle \neq 0$, and we also have
\begin{align}
\langle \mu, \gamma^\vee \rangle + \langle \mu, \iota(\gamma^\vee) \rangle = \langle \beta, \gamma^\vee \rangle \langle \mu, \beta^\vee \rangle
\end{align}
In particular, at least one of $\langle \mu, \gamma^\vee \rangle$ or $\langle \mu, \iota(\gamma^\vee) \rangle$ must have the same sign as $\langle \mu, \beta^\vee \rangle$ (where we interpret zero to be positive for this purpose). So a majority of the elements $\gamma^\vee \in \Inv{s_\beta}$ must have the property that $\langle \mu, \gamma^\vee \rangle \geq 0$. The other case follows similarly.

\end{proof}

\begin{proof}[Proof of Theorem~\ref{thm:the-two-orders-coincide}] 
Let $\pi^\mu w \in \cW_\cT$, and let $\beta^\vee + n \pi$ be a positive double affine root. And furthermore, suppose that $\pi^\mu w s_{\beta^\vee + n \pi} \in \cW_\cT$. Then we need to show that if  
\begin{align}
  \pi^\mu w(\beta^\vee + n \pi) > 0
\end{align}
then, 
\begin{align}
  \label{eq:100}
 \ell(\pi^\mu w s_{\beta^\vee + n \pi}) >  \ell(\pi^\mu w) 
\end{align}
and to show the similar statement where the inequality signs are reversed.
Let us consider the case when
\begin{itemize}
\item $\beta^\vee > 0$
\item $\pi^\mu w(\beta^\vee + n \pi) > 0$
\end{itemize}
In this case, 
\begin{align}
  s_{\beta^\vee + n \pi}= \pi^{n\beta}s_{\beta}
\end{align}
We have 
\begin{align}
\pi^\mu w(\beta^\vee + n \pi) = w(\beta^\vee) + (n + \langle(w(\beta^\vee), \mu \rangle ) \pi
\end{align}
and we also have  
\begin{align}
  \pi^\mu w s_{\beta^\vee + n \pi} = \pi^{\mu + n w(\beta)}w s_{\beta}
\end{align}

Therefore, we have
\begin{align}
  \label{eq:14}
   n + \langle(w(\beta^\vee), \mu \rangle \geq 0
\end{align}
If the inequality is strict, using Lemma \ref{lem:big-length-lemma-2} we compute that $\ellbig(\pi^\mu w s_{\beta^\vee + n \pi}) > \ellbig(\pi^\mu w)$, which implies our desired result.

So all that remains is the case where
\begin{align}
  n = -\langle(w(\beta^\vee), \mu \rangle
\end{align}
In this case, we have 
\begin{align}
\pi^\mu w(\beta^\vee + n \pi) = w(\beta^\vee)
\end{align}
In particular, we have $w(\beta^\vee) > 0$, and we have 
\begin{align}
  \pi^\mu w s_{\beta^\vee + n \pi} = \pi^{s_{w(\beta)}(\mu)}w s_{\beta} = s_{w(\beta)}\pi^\mu w
\end{align}

By repeated use of Lemma \ref{lem-length-recursion-left}, we see that 
\begin{align}
 \ell( s_{w(\beta)}\pi^\mu w ) = \ell(\pi^\mu w) + \varepsilon \cdot \left( \left|\{ \gamma \in \Inv{s_{w(\beta)}} : \langle \mu, \gamma \rangle \geq 0 \}\right| - \left|\{ \gamma \in \Inv{s_{w(\beta)}} : \langle \mu, \gamma \rangle < 0 \}\right| \right)
\end{align}
By Lemma \ref{lem:inversion-set-lemma}, we know that the sign of 
\begin{align}
\left( \left|\{ \gamma \in \Inv{s_{w(\beta)}} : \langle \mu, \gamma \rangle \geq 0 \}\right| - \left|\{ \gamma \in \Inv{s_{w(\beta)}} : \langle \mu, \gamma \rangle < 0 \}\right| \right) 
\end{align}
is the same as the sign of $\langle \mu, w(\beta^\vee) \rangle$, which is positive (recall here that when $\langle \mu, w(\beta^\vee) \rangle = 0$ we also say it has ``positive'' sign for this purpose). So we have that
\begin{align}
 \ell(\pi^\mu w s_{\beta^\vee + n \pi}) >  \ell(\pi^\mu w) 
\end{align}
as desired. The other cases follow by a similar argument.
\end{proof}

\section{Some remarks when $\bG$ is finite-type}

\newcommand{\ellcox}{\ell_\text{cox}}

In this section, we will consider the case when $\bG$ is finite-type and compare the usual development of the Bruhat order and length function with the proofs given in the previous section. For simplicity, let us additionally assume that $\bG$ is simply connected. In this case, $\cW_P = \cW_Q = \cW_\cT$. This group is usually denoted $\Waff$ for the \emph{(single) affine Weyl group}, and notably $\Waff$ is a Coxeter group. There is a general Coxeter-group notion of Bruhat order on $\Waff$, which is graded by the usual length function $\ellcox$ taking values in non-negative integers.

Moreover, we can study the Iwahori-Hecke algebra of functions on $G(F)$ that are biinvariant for the action of the Iwahori subgroup and supported on finitely many double cosets. We have the basis given by indicator functions of double cosets $\{ T_x \mid x \in \Waff \}$, and we have the usual Iwahori-Matsumoto relations.

\begin{Proposition}{\cite[Corollary 3.6]{im}}
\begin{align}{\label{eqn:ordinary-iwahori-matsumoto-formula}}
T_{\pi^{\mu} w s_i} = \left\{
\begin{array}{c l}      
   T_{\pi^\mu w} T_i \text{ if } \ell(\pi^{\mu} w s_i) = \ell(\pi^{\mu} w) + 1  \\
    T_{\pi^\mu w} T_i^{-1}  \text{ if } \ell(\pi^{\mu} w s_i) = \ell(\pi^{\mu} w) - 1 
\end{array}\right.
\end{align}
\end{Proposition}
We also have the following well-known facts.
\begin{Proposition}
Let $\lambda$ be a dominant coweight. Then we have 
\begin{align}
  \ellcox(\pi^\lambda) = \langle 2\rho^\vee,\lambda\rangle
\end{align}
In addition, for $w \in W$, we have 
\begin{align}
  \ellcox(\pi^{w(\lambda)}) = \ellcox(\pi^\lambda)
\end{align}
\end{Proposition}

In addition to this classical story, the methods of section \ref{sec:bruhat-order-section} apply when $\bG$ is finite-type. In particular, by the Braverman-Kazhdan-Patnaik definition of the Bruhat order, we see that the Bruhat order considered in section \ref{sec:bruhat-order-section} agrees with the general Coxeter-group definition of Bruhat order on $\Waff$. As a consequence, the length function $\ell : \Waff \rightarrow \ZZ \oplus \ZZ \varepsilon$ also gives a grading of the Bruhat order. 

\newcommand{\RR}{\mathbb{R}}

Let $t \in \RR$, and let $\ell_t: \Waff \rightarrow \RR$ be the composed map $\Waff \rightarrow \ZZ \oplus \ZZ \varepsilon \rightarrow \RR$, where the first map is $\ell$, and the second map is given by setting $\varepsilon$ equal to $t$. Then looking at the usual Iwahori-Matsumoto relation, we see that $\ell_1=\ellcox$, the usual Coxeter length function on $\Waff$. From this we obtain the following fact.

\begin{Proposition}
  For all $t \in (0,1]$, $\ell_t$ is a grading for the Bruhat order on $\Waff$.
\end{Proposition}
\begin{proof}
Suppose $x \in \Waff$, $r \in \Waff$ is a reflection corresponding to real affine root, and that 
\begin{align}
  \label{eq:11}
 \ell_1(xr) > \ell_1(x) 
\end{align}
By the results of section \ref{sec:bruhat-order-section}, we also have that 
\begin{align}
  \label{eq:12}
  \ell(xr) > \ell(x)
\end{align}
Then we need to prove that  $\ell_t(xr) > \ell_t(x)$ for all $t \in (0,1]$. 
By \eqref{eq:11}, we have 
\begin{align}
  \label{eq:13}
  \ellbig(xr) - \ellbig(x) > \ellsmall(xr) - \ellsmall(x)
\end{align}

There are two cases.
\begin{itemize}
\item  $\ellsmall(xr) - \ellsmall(x) > 0$:

In this case, we have   $ \ellbig(xr) - \ellbig(x) > \ellsmall(xr) - \ellsmall(x) \geq t\cdot (\ellsmall(xr) - \ellsmall(x))$ for all $t \in (0,1]$, which is equivalent to our desired result.
\item $\ellsmall(xr) - \ellsmall(x) \leq 0$:

In this case, by \eqref{eq:12}, we have $ \ellbig(xr) - \ellbig(x) > 0$, from which we have $ \ellbig(xr) - \ellbig(x) > 0 \geq t\cdot (\ellsmall(xr) - \ellsmall(x))$ for all $t \in (0,1]$.
\end{itemize}
\end{proof}

We can ask how much of this carries through when $\bG$ is infinite-type.
\begin{Question}
Let $\bG$ be an infinite-type Kac-Moody group. Is $\ell_1$ a grading for the Bruhat order on $\cW_\cT$? 
\end{Question}
\noindent It would not be very surprising if the answer to this is no. In that case, we can still ask the following weaker question.
\begin{Question}
 Let $\bG$ be an infinite-type Kac-Moody group. Can the Bruhat order on $\cW_\cT$ be graded by $\ZZ$? 
\end{Question}


\end{document}